\documentclass{article}

\usepackage{cite}
\usepackage{graphicx}
\usepackage{algorithm}
\usepackage{algorithmic}
\usepackage{epsfig}
\usepackage{epstopdf}
\usepackage{amssymb,amsmath}
\usepackage{amsmath}
\usepackage{amsthm}
\usepackage{amsfonts}
\usepackage{subfigure}
\usepackage{psfrag}
\usepackage{rotating}
\usepackage{latexsym}
\usepackage{dsfont}
\usepackage{stmaryrd}
\usepackage{amsthm}
\usepackage{amssymb}
\usepackage{amsmath}

\usepackage[force,almostfull]{textcomp}
\usepackage{fancyvrb}
\usepackage{textcomp}
\usepackage{url}
\usepackage{ifdraft}
\usepackage{url}
\usepackage{multirow}
\usepackage{rotating}
\usepackage{xspace}
\usepackage{array}
\usepackage{multido}
\usepackage{enumerate}
\usepackage[latin1]{inputenc}
\newlength\figureheight 
\newlength\figurewidth 
\usepackage{graphics}
\usepackage{pgfplots}
\pgfplotsset{compat=newest}
\pgfplotsset{plot coordinates/math parser=false}

\usepackage{scalefnt}
\newtheorem{specialcasecounter}{Theorem}
\newtheoremstyle{specialcasestyle}{1mm}{1mm}{\upshape}{}{\bfseries\upshape}{.}{0mm}{}
\theoremstyle{specialcasestyle}

\newtheorem{lem}{Lemma}
\newtheorem{prop}{Proposition}

\begin{document}

\title{On the Efficient Simulation of the Left-Tail of the Sum of Correlated Log-normal Variates}
\author{Mohamed-Slim Alouini, Nadhir Ben Rached, Abla Kammoun, \\and Raul Tempone\\\thanks{The authors are in Computer, Electrical and Mathematical Science and Engineering (CEMSE) Division, King Abdullah University of Science and Technology (KAUST), Thuwal, Makkah Province, Saudi Arabia , and also members of the KAUST Strategic Research Initiative on Uncertainty Quantification in Science and Engineering (SRI-UQ) (e-mail: {slim.alouini,nadhir.benrached, abla.kammoun, raul.tempone}@kaust.edu.sa).
}
}
\date{}
\maketitle
\thispagestyle{empty}

\begin{abstract}
The sum of Log-normal variates is encountered in many challenging applications such as in performance analysis of wireless communication systems and in financial engineering. Several approximation methods have been developed in the literature, the accuracy of which is not ensured in the tail regions. These regions are of primordial interest wherein small probability values have to be evaluated with high precision. Variance reduction techniques are known to yield accurate, yet efficient, estimates of small probability values. Most of the existing approaches, however,  have considered the problem of estimating the right-tail of the sum of Log-normal random variables (RVS). In the present work, we consider instead the estimation of the left-tail of the sum of correlated Log-normal variates with Gaussian copula under a mild assumption on the covariance matrix. We propose an estimator combining an existing mean-shifting importance sampling approach with a control variate technique. The main result is that the proposed estimator has an asymptotically vanishing relative error which represents a major finding in the context of the left-tail simulation of the sum of Log-normal RVs. Finally, we assess by various simulation results the performances of the proposed estimator compared to existing estimators.
\end{abstract}
\section{Introduction}
The Log-normal distribution is encountered in many applications such as in financial engineering \cite{CIS-230627}. In performance analysis of wireless communications systems, it has shown a good fit to  realistic propagation channels \cite{Stuber:2001:PMC:368633,Ghavami,journals/twc/NavidpourUK07}. Therefore, investigating the distribution of sums of Log-normal random variables (RVs) is of primordial practical interest. In fact, the pricing of Asian or basket options is closely related to the distribution of the sum of Log-normal variates \cite{CIS-230627}. A further application where the problem of finding the distribution of sums of Log-normal variates is encountered is in the evaluation of the value at risk, defined as (1-$\alpha$) quantile of the loss distribution, for a sufficiently small value of $\alpha$ \cite{asmussen2016exponential}. In wireless communication systems, the distribution of sums of Log-normal RVs is of major practical interest in the problem of evaluating the outage probability at the output  of receivers with diversity techniques \cite{7328688}.

The distribution of the sum of Log-normal variates is not known in a closed form. This has lead researchers to propose various approximation techniques such as in \cite{citeulike:6297231,citeulike:7151841,4784348,4814351,1275712}. However, the accuracy of these approximations is not ensured for all scenarios especially in the tail regions, i.e. the region we are interested in, as illustrated in \cite{1275712}.

Variance reduction Monte Carlo (MC) methods constitute good alternatives to efficiently estimate tail probabilities of the distribution of sums of Log-normal RVs. The literature of estimating the distribution of sums of Log-normal variates is very rich. However,  most of them have concentrated in the right-tail region instead of the left-tail one, which is the region of interest in this work. For instance, a non exhaustive list of efficient variance reduction techniques have been developed in \cite{reference_1,BenRached2016,Juneja:2002:SHT:566392.566394} to estimate the right-tail of the distribution of sums of independent Log-normal RVs. In the correlated setting as well, efficient simulation approaches of the right-tail of sums of correlated Log-normal variates have been proposed in \cite{DBLP:journals/anor/AsmussenBJR11,DBLP:conf/wsc/BlanchetJR08,CIS-358019}. 

The estimation of the distribution of the left-tail of sums of Log-normal RVs has received less interest than that of the right-tail region. Estimating the left-tail of the distribution of sums of Log-normal variates is motivated by the problem of evaluating outage probabilities in the performance analysis of wireless communication systems operating over a Log-normal fading environment. It is only recently that researchers have accorded some interest to  the estimation of the left-tail region \cite{7328688,asmussen2016exponential,gulisashvili2016}. In fact, the authors in \cite{asmussen2016exponential} have used the well-known importance sampling (IS) technique; namely the exponential twisting approach \cite{179349}, and have proved that the proposed estimator achieves the asymptotic optimality property under the independent and identically distributed (i.i.d) assumption. However, given that the application of the exponential twisting technique requires the knowledge of the moment generating function (MGF) which is out of reach for the Log-normal distribution, the authors in \cite{asmussen2016exponential} have considered instead an estimator of the MGF \cite{Laplace}. In \cite{7328688}, two unified IS approaches have been proposed using the well-known hazard rate twisting technique \cite{BenRached2016,Juneja:2002:SHT:566392.566394}. In particular, for the Log-normal setting, the first estimator was shown to achieve the asymptotic optimality criterion under the assumption of independent and not necessarily identically distributed sums of Log-normal RVs. The asymptotic optimality property holds again using the second IS scheme under the i.i.d assumption. Comparisons between the estimator of \cite{asmussen2016exponential} and the two latter estimators have been performed in \cite{7328688}. The sole work the authors are aware of it that deals with the efficient simulation of the left-tail of the sum of correlated Log-normal variates is in \cite{gulisashvili2016}. In fact, an IS estimator has been proposed by shifting the mean of the corresponding multivariate normal distribution. This estimator was shown to achieve the asymptotic optimality property under a mild assumption. 

We aim in the present paper to further improve the mean-shifting IS estimator of \cite{gulisashvili2016}. More precisely, we consider the left-tail simulation of sums of correlated Log-normal variates, that is we aim to estimate the probability that a sum of correlated Log-normal variates is less than a sufficiently small threshold, and we provide an improved estimator of the IS scheme proposed in \cite{gulisashvili2016}. Our methodology is based on the asymptotic behavior results provided in \cite{gulisashvili2016} and considers the case where the left-tail of the distribution of the sum is determined by only one dominant component. In this setting, we improve the IS estimator of \cite{gulisashvili2016} by combining it with a control variate type of variance reduction technique. The introduced control variate is function of the dominant component that characterizes the tail behavior of the left-tail of the sum. The main result of the present work is that the improved estimator using the proposed control variate technique has the asymptotically vanishing relative error property which is the most desired property in the context of rare event simulations. Such a result represents a relevant contribution for the problem of estimating the left-tail of sums of Log-normal RVs since as it was mentioned above the existing estimators in the corresponding literature were only proved to achieve a weaker property; namely the asymptotic optimality criterion \cite{7328688,asmussen2016exponential,gulisashvili2016}. Simulation results show that our proposed approach yields a substantial amount of variance reduction compared to the mean-shifting IS approach of \cite{gulisashvili2016}, a reduction that increases as we decrease the probability of interest, i.e. as we decrease the threshold value to zero. 

The rest of the paper is organized as follows. In section II, we describe the problem setting and we review the concepts of IS and control variate. Section III is devoted to presenting the IS scheme of \cite{gulisashvili2016}. The main idea of the paper is provided in Section IV where we show how to improve the IS approach of \cite{gulisashvili2016} through the use of control variate. In the same section, the main result proving the asymptotically vanishing relative error property is provided. Finally, some selected simulation results are provided in Section V to assess the performance of the proposed approach which achieves a substantial amount of computational gain over the mean shifting IS of \cite{gulisashvili2016}.

\section{Problem Setting}
We consider a random vector $\bold{Y}=(Y_1,Y_2,\cdots,Y_N)^t$ with $N$-dimensional multivariate normal distribution with mean vector $\boldsymbol{\mu}$ and covariance matrix $\bold{\Sigma}$. We assume that $\bold{\Sigma}$ is positive definite. The probability density function (PDF) of the random vector $\bold{Y}$ is given as follows
\begin{align}
f(\boldsymbol{y})=\frac{\exp \left (-\frac{1}{2} \left (\boldsymbol{y}-\boldsymbol{\mu} \right )^t\bold{\Sigma}^{-1} \left ( \boldsymbol{y}-\boldsymbol{\mu}\right ) \right )}{\sqrt{(2\pi)^{N}|\bold{\Sigma}|}},
\end{align}
where $|\bold{\Sigma}|$ denotes the determinant of the matrix $\bold{\Sigma}$ and $\boldsymbol{y}^t$ is the transpose of the vector $\boldsymbol{y}$. Similarly to \cite{gulisashvili2016}, we define the following quantities. The elements of $\bold{\Sigma}$ and its inverse $\bold{\Sigma}^{-1}$ will be denoted by $\Sigma_{ij}$ and $\Sigma_{ij}^{-1}$, respectively. Moreover, we define $A_k=\sum_{j=1}^{N}{\Sigma_{kj}^{-1}}$. Let $\bar{ \boldsymbol{w}}$ be the unique solution of 
\begin{align}\label{mini}
\min_{\Delta_{N}}{\boldsymbol{w}^t \bold{\Sigma}}\boldsymbol{w}
\end{align}
with $\Delta_N= \{\boldsymbol{w} \in \mathbb{R}^N \text{  such that  } w_i \geq 0, i=1,2,\cdots,N, \text{ and} \sum_{i=1}^{N}{w_i}=1\} $. Let us denote by $\bar n$ the cardinal of the set $\{1,2,\cdots,N, \text{  such that  } \bar{w}_i \neq 0 \}$ and $\bar{I}=\{ \bar{k}(1),\cdots,\bar{k}(\bar{n})\}$ the corresponding indexes. Then, we denote by $\bar{\boldsymbol{\mu}} \in \mathbb{R}^{\bar{n}}$ the vector with entries $\bar{\mu}_i=\mu_{\bar{k}(i)}$ and $\bold{\bar{\Sigma}}$ the $\bar{n} \times \bar{n}$ matrix with elements $\bar{\Sigma}_{ij}=\Sigma_{\bar{k}(i)\bar{k}(j)}$. Moreover, we denote by $\bar{\bold{\Sigma}}^{-1}$ the inverse of $\bar{\bold{\Sigma}}$ with elements $\bar{\Sigma}_{ij}^{-1}$. Finally, the row sums of $\bar{\bold{\Sigma}}^{-1}$ is denoted by $\bar{A_k}=\sum_{j=1}^{\bar{n}}{\bar{\Sigma}_{kj}^{-1}}$. As it was mentioned in \cite{gulisashvili2016}, we assume without loss of generality that $\bar{I}=\{1,2,\cdots,\bar{n}\}$. 

Let us consider the Log-normal random vector $\bold{X}=(X_1,X_2,\cdots,X_N)^t$ such that $Y_i=\log (X_i)$, $i=1,2,\cdots,N$. Our aim is to efficiently estimate the left-tail of the sum of $X_i$'s, that is the probability that the sum of correlated Log-normal RVs with Gaussian copula falls below a sufficiently small threshold:
\begin{align}
\alpha (\gamma_{th})= P_f \left ( \sum_{i=1}^{N}{X_i} \leq \gamma_{th}\right ).
\end{align}
As it was mentioned previously, the above quantity is of paramount practical interest in the performance evaluation of wireless communication systems as it corresponds to the probability that the communication system is in outage, i.e. the probability that the system fails to operate correctly. The simplest method to estimate $\alpha(\gamma_{th})$ is to use the naive MC estimator. However, it is well-known that the naive estimator is computationally expensive when estimating rare events, i.e. region of sufficiently small values of $\alpha(\gamma_{th})$ which is the region of interest of the present work. 

To construct computationally efficient estimators, variance reduction techniques represent good alternatives \cite{opac-b1132466}. When appropriately employed, variance reduction techniques are known to yield accurate estimate with a smaller computational effort than the naive MC sampler. In the present work, we will consider two instances of variance reduction techniques which are IS and control variate. 
\subsection{Review of IS}
IS has been extensively employed for the efficient simulation of rare events \cite{BenRached2016,Juneja:2002:SHT:566392.566394,asmussen2016exponential,DBLP:journals/anor/AsmussenBJR11}. The main idea of IS is to introduce an IS distribution $g(\cdot)$ under which sampling is performed instead of the original PDF $f(\cdot)$. In fact, the probability of interest $\alpha(\gamma_{th})$ could be re-written as follows 
\begin{align}
\nonumber \alpha(\gamma_{th})&=\mathbb{E}_{f} \left [\bold{1}_{\left (\sum_{i=1}^{N}{\exp \left (Y_i \right )} \leq \gamma_{th} \right )} \right ]\\
&= \mathbb{E}_{g} \left [\bold{1}_{\left (\sum_{i=1}^{N}{\exp \left (Y_i \right )} \leq \gamma_{th} \right )} L\left (Y_1,Y_2,\cdots, Y_N \right ) \right ],
\end{align}
where $\bold{1}_{\left (\cdot \right )}$ denotes the indicator function, whereas $\mathbb{E}_{f}\left [ \cdot\right ]$ and $\mathbb{E}_{g}\left [ \cdot\right ]$ are the expectation operators under the PDFs $f(\cdot)$ and $g(\cdot)$, respectively. $L(\cdot)$ is the likelihood ratio which is defined as
\begin{align}
L\left (Y_1,Y_2,\cdots, Y_N \right )=\frac{f\left (Y_1,Y_2,\cdots, Y_N \right )}{g\left (Y_1,Y_2,\cdots, Y_N \right )}.
\end{align}
Following the above probability change of measure, the IS estimator is defined as:
\begin{align}
\hat \alpha_{IS}(\gamma_{th})=\frac{1}{M}\sum_{k=1}^{M}{\bold{1}_{\left (\sum_{i=1}^{N}{\exp \left (Y_i^{(k)} \right )} \leq \gamma_{th} \right )} L\left (Y_1^{(k)},Y_2^{(k)},\cdots, Y_N^{(k)} \right )},
\end{align}
where $\{(Y_1^{(k)},Y_2^{(k)},\cdots,Y_N^{(k)})^t\}_{k=1}^{M}$ are independent realizations of the random vector $\bold{Y}$ under the PDF $g(\cdot)$. The crucial step when using IS is the choice of the IS distribution $g(\cdot)$ that results in a variance reduction. In fact, a good IS distribution encourages the sampling of important realizations, that is samples that belong to the rare set $\{\sum_{i=1}^{N}{\exp \left ( Y_i\right )}\leq \gamma_{th}\}$, and also tries to maintain the likelihood ratio $L(\cdot)$ constant in the rare region. In the following section, we will review the IS technique developed in \cite{gulisashvili2016} for the estimation of $\alpha(\gamma_{th})$.
\subsection{Review of Control Variate}
Control variate is also a variance reduction technique that can, if adequately used, yield a substantial amount of variance reduction. This method could be combined with IS in order to further reduce the variance. Suppose we want to estimate $\alpha=\mathbb{E}_g \left [ T_{\gamma_{th}}\left (\bold{Y} \right )\right ]$ where $T_{\gamma_{th}}\left (\bold{Y} \right )= \bold{1}_{\left (\sum_{i=1}^{N}{\exp \left (Y_i \right )} \leq \gamma_{th} \right )} L\left (Y_1,\cdots, Y_N \right )$ is an IS estimator. Let $Z_{\gamma_{th}} \left ( \bold{Y}\right )$ be a control variate with a known expected value $P(\gamma_{th})$. The idea of control variate technique is to consider the following estimator of $\alpha(\gamma_{th})$
\begin{align}
T^{'}_{\gamma_{th}} \left (\bold{Y} \right )=T_{\gamma_{th}} \left (\bold{Y} \right )+\beta \left (Z_{\gamma_{th}} \left (\bold{Y} \right )-P(\gamma_{th}) \right),
\end{align}
where $\beta \in \mathbb{R}$. Obviously, the above estimator is an unbiased estimator of $\alpha(\gamma_{th})$. The variance of $T^{'}_{\gamma_{th}} \left (\bold{Y} \right )$ is given by
\begin{align}\label{variance}
\mathrm{var}_g \left [T^{'}_{\gamma_{th}} \left (\bold{Y} \right ) \right ]&=\mathrm{var}_g \left [T_{\gamma_{th}} \left (\bold{Y} \right ) \right ]+2 \beta \nonumber \mathrm{cov}_g \left [ T_{\gamma_{th}} \left (\bold{Y} \right ),Z_{\gamma_{th}} \left ( \bold{Y} \right )\right ]\\
&+\beta^2 \mathrm{var}_g \left [Z_{\gamma_{th}} \left ( \bold{Y}\right ) \right ].
\end{align}
Hence, the optimal value of $\beta$ is the value that minimizes the variance of $T^{'}_{\gamma_{th}} \left (\bold{Y} \right )$ and given, through a simple computation, by
\begin{align}\label{opt_beta}
\beta^*=-\frac{\mathrm{cov}_g \left [T_{\gamma_{th}} \left (\bold{Y} \right ),Z_{\gamma_{th}} \left ( \bold{Y}\right ) \right ]}{\mathrm{var}_g \left [Z_{\gamma_{th}} \left ( \bold{Y}\right ) \right ]}.
\end{align}
By plugging the optimal value $\beta^*$ into (\ref{variance}), we easily get that the variance of $T^{'}_{\gamma_{th}} \left (\bold{Y} \right )$ is given by
\begin{align}\label{optimal_reduction}
\mathrm{var}_g \left [T^{'}_{\gamma_{th}} \left (\bold{Y} \right ) \right ]= \left (1-\rho_{T_{\gamma_{th}} \left (\bold{Y} \right ),Z_{\gamma_{th}} \left ( \bold{Y} \right )}^2 \right ) \mathrm{var}_g \left [ T_{\gamma_{th}} \left (\bold{Y} \right )\right ],
\end{align}
where $\rho_{T_{\gamma_{th}} \left (\bold{Y} \right ),Z_{\gamma_{th}} \left ( \bold{Y} \right )}$ is the correlation coefficient between the two RVs $T_{\gamma_{th}} \left ( \bold{Y}\right )$ and $Z_{\gamma_{th}} \left (\bold{Y} \right )$. From the above result, we conclude that, in order to further reduce the variance of $T_{\gamma_{th}} \left (\bold{Y} \right )$, the RV $Z_{\gamma_{th}} \left ( \bold{Y}\right )$ has to be selected such that it is highly correlated with the RV $T_{\gamma_{th}} \left (\bold{Y} \right )$. In the next section, we will describe how the RV $Z_{\gamma_{th}} \left (\bold{Y} \right )$ is selected in order to achieve a substantial amount of variance reduction. The estimator of $\alpha(\gamma_{th})$ following  the control variate technique with any value of $\beta$ is as follows
\begin{align}
\hat \alpha_{IS-CV}(\gamma_{th})=\frac{1}{M}\sum_{k=1}^{M}{\left ( T_{\gamma_{th}} \left (\bold{Y^{(k)}} \right )+\beta \left (Z_{\gamma_{th}} \left ( \bold{Y^{(k)}}\right )-P(\gamma_{th}) \right) \right )}.
\end{align}
Note that the optimal value $\beta^*$ is generally unknown and has to be estimated via sample covariance and sample variance estimators. However, it is worth mentioning that estimating $\beta^*$ using the same simulated data used in getting the estimate introduces some dependence in the above estimate. However, it was shown in \cite{glasserman2004monte} that working with the estimated value of $\beta^*$ yield the same estimator's performances as working with $\beta^*$ for large number of samples. 
\subsection{Performance Metrics}
Many criteria have been used in practice in order to measure the efficiency of an unbiased estimator such as the asymptotically vanishing relative error, the bounded relative error, and the asymptotic optimality \cite{opac-b1123521}. 
For any estimator $T^{'}_{\gamma_{th}} \left (\bold{Y} \right ) $ of $\alpha(\gamma_{th})$ with $\bold{Y}$ is distributed according to the PDF $g(\cdot)$, we have from the non-negativity of the variance of $T^{'}_{\gamma_{th}} \left (\bold{Y} \right ) $
\begin{align}
\mathbb{E}_{g} \left [ T^{'2}_{\gamma_{th}} \left (\bold{Y} \right ) \right ] \geq \alpha^2(\gamma_{th}).
\end{align}
Using the fact that $\log \left (\alpha(\gamma_{th}) \right )<0$, we get
\begin{align}
\frac{\log \left (\mathbb{E}_{g} \left [ T^{'2}_{\gamma_{th}} \left (\bold{Y} \right ) \right ] \right )}{\log \left (\alpha(\gamma_{th})\right )} \leq 2.
\end{align}
We say that the estimator $T^{'}_{\gamma_{th}} \left (\bold{Y} \right )$ achieves the asymptotic optimality property if 
\begin{align}\label{asym_op}
\lim_{\gamma_{th} \rightarrow 0}{\frac{\log \left (\mathbb{E}_{g} \left [ T^{'2}_{\gamma_{th}} \left (\bold{Y} \right ) \right ] \right )}{\log \left (\alpha(\gamma_{th}) \right )} }=2.
\end{align}
An equivalent definition of the asymptotic optimality criterion is: $\forall \epsilon >0 $, we have
\begin{align}
\lim_{\gamma_{th} \rightarrow 0}{\frac{\mathrm{var}_{g} \left [ T^{'}_{\gamma_{th}} \left (\bold{Y} \right ) \right ]}{\alpha^{2-\epsilon}(\gamma_{th})} }=0.
\end{align}
Two interesting interpretations could be deduced from the asymptotic optimality property. First of all, when $\alpha(\gamma_{th})^2 \rightarrow 0$ with an exponential rate, the second moment of $T^{'}_{\gamma_{th}} \left (\bold{Y} \right )$ converges to zero with the same exponential rate. This is the best exponential rate that the second moment may converge with. Second, when the asymptotic optimality property holds, the number of simulation runs $M$ required to meet a fixed accuracy requirement satisfies $M=o \left ( \alpha(\gamma_{th})^{-\epsilon}\right)$ for all $\epsilon >0$. Such a result ensures that an estimator with the asymptotic optimality criterion will certainly yield a substantial amount of variance reduction compared to naive MC simulation which requires a number of runs of the order of $\alpha(\gamma_{th})^{-1}$ to meet the same accuracy requirement. 

A stronger criterion than the asymptotic optimality is the bounded relative error. In fact, this property holds when
\begin{align}
\limsup_{\gamma_{th} \rightarrow 0}{\frac{\mathrm{var}_g \left [ T^{'}_{\gamma_{th}} \left ( \bold{Y}\right )\right ]}{\alpha^2(\gamma_{th})}} < +\infty.
\end{align}  
When the bounded relative error property holds, the number of simulation runs needed to meet a fixed accuracy requirement remains bounded regardless of how small $\alpha(\gamma_{th})$ is. 

A further stronger criterion is the asymptotically vanishing relative error which holds when
\begin{align}
\limsup_{\gamma_{th} \rightarrow 0}{\frac{\mathrm{var}_g \left [ T^{'}_{\gamma_{th}} \left ( \bold{Y}\right )\right ]}{\alpha^2(\gamma_{th})}}=0.
\end{align}
This means that the number of samples is getting smaller as we decrease the probability of interest while ensuring a fixed accuracy requirement. 
\section{Mean-Shifting Approach}
In this section, we review the mean-shifting IS scheme proposed in \cite{gulisashvili2016}. The main idea is to consider an IS distribution resulting from shifting the mean of the multivariate normal distribution. More precisely, the IS PDF is chosen to be a multivariate normal with mean vector $\boldsymbol{\mu}+\bold{\Lambda}$ and covariance matrix $\bold{\Sigma}$ with $\bold{\Lambda} \in \mathbb{R}^{N}$:
\begin{align}
g(\boldsymbol{y})=\frac{\exp \left (-\frac{1}{2} \left (\boldsymbol{y}-\boldsymbol{\mu}-\bold{\Lambda} \right )^t\bold{\Sigma}^{-1} \left ( \boldsymbol{y}-\boldsymbol{\mu}-\bold{\Lambda}\right ) \right )}{\sqrt{(2\pi)^{N}|\bold{\Sigma}|}}.
\end{align}
The likelihood ratio following this IS scheme is given as follows
\begin{align}
L \left (Y_1,Y_2,\cdots,Y_N \right )=\exp \left (-\bold{\Lambda}^t \bold{\Sigma}^{-1} (\bold{Y}-\boldsymbol{\mu})+\frac{1}{2} \bold{\Lambda}^t\bold{\Sigma}^{-1}\bold{\Lambda}\right).
\end{align}
Hence, the probability of interest $\alpha(\gamma_{th})$ is re-written as
\begin{align}
\nonumber \alpha(\gamma_{th})&= \mathbb{E}_{g} \left [\exp \left (-\bold{\Lambda}^t \bold{\Sigma}^{-1} (\bold{Y}-\boldsymbol{\mu})+\frac{1}{2} \bold{\Lambda}^t\bold{\Sigma}^{-1}\bold{\Lambda}\right) \bold{1}_{\left (\sum_{i=1}^{N}{\exp \left (Y_i \right )} \leq \gamma_{th} \right )} \right ]\\
&=\mathbb{E}_{f} \left [ \exp \left (-\bold{\Lambda}^t \bold{\Sigma}^{-1} (\bold{Y}-\boldsymbol{\mu})-\frac{1}{2} \bold{\Lambda}^t\bold{\Sigma}^{-1}\bold{\Lambda}\right) \bold{1}_{\left (\sum_{i=1}^{N}{\exp \left (Y_i+\Lambda_i \right )} \leq \gamma_{th} \right )} \right ].
\end{align}
By simple computation, it was shown in \cite{gulisashvili2016} that the second moment of $T_{\gamma_{th}}\left ( \bold{Y}\right)$ is given by
\begin{align}\label{second_moment}
\mathbb{E}_{g} \left [T_{\gamma_{th}}^2 \left (\bold{Y} \right ) \right ]=\exp \left (\bold{\Lambda}^t \bold{\Sigma}^{-1} \bold{\Lambda} \right ) P_f \left (\sum_{i=1}^{N}{\exp \left (Y_i-\Lambda_i \right )} \leq \gamma_{th}\right ).
\end{align}

The remaining step is to determine the value of $\bold{\Lambda}$ that guarantees a variance reduction compared to naive MC estimator. Obviously, the optimal value is to minimize the second moment of the RV $T_{\gamma_{th}} \left (\bold{Y} \right )$ with respect to $\bold{\Lambda}$. However, the second moment in (\ref{second_moment}) is not known explicitly. The authors in \cite{gulisashvili2016} have then proposed to find $\bold{\Lambda}$ that minimizes an asymptotic equivalent given by replacing the probability in (\ref{second_moment}) by an equivalent expression given in \cite{gulisashvili2016}. The value of $\bold{\Lambda}$ proposed in \cite{gulisashvili2016} is as follows
\begin{align}\label{opt_lam}
\Lambda_{k}^{*}=\sum_{i,j=1}^{\bar{n}}{\Sigma_{ki}\bar{\Sigma}^{-1}_{ij} \left ( \log \left (x \right ) -\log \left (\frac{\bar{A}_1+\cdots+\bar{A}_{\bar{n}}}{\bar{A}_j} \right )-\bar{\mu}_{j}\right )}.
\end{align}
With this value of $\bold{\Lambda}$, it was proven in \cite{gulisashvili2016} that the second moment of the IS estimator satisfies:
\begin{align}\label{bound_sec_moment}
\mathbb{E}_{g} \left [ T_{\gamma_{th}}^2 \left (\bold{Y} \right ) \right ] \leq C \alpha^2(\gamma_{th}) \left (\log \left ( \frac{1}{\gamma_{th}}\right ) \right )^{\bar{n}},
\end{align}
where $C$ is a constant that does not depend on $\gamma_{th}$. Thus, the asymptotic optimality property (\ref{asym_op}) holds. This result is based on the asymptotic behavior of $\alpha(\gamma_{th})$ as $\gamma_{th}  \rightarrow 0$, see Theorem 1 of \cite{gulisashvili2016}.

\section{Improved Algorithm Using Control Variate}
The results that will be shown in the present section represent the main contribution of our work. In fact, we aim here to combine the IS scheme presented in the previous section with a control variate technique to further achieve a variance reduction. We assume that the following assumption holds:\newline
\textbf{Assumption A}: there exist an index $i \in \{1,2,\cdots,N \}$ such that $\sqrt{\Sigma_{ii}} <\rho_{ij} \sqrt{\Sigma_{jj}}$ for all $j\neq i$, where $\rho_{ij}=\frac{\Sigma_{ij}}{\sqrt{\Sigma_{ii}\Sigma_{jj}}}$ denotes the correlation coefficient between $Y_i$ and $Y_j$.
\newline With no loss of generality, we may suppose that $i=1$. Under the above assumption, it was proven in \cite{gulisashvili2016} that the asymptotic behavior of the left-tail  of the sum is dominated by only one component that corresponds to the index $i=1$ in assumption A, that is
\begin{align}\label{dominant}
\nonumber \alpha(\gamma_{th}) &\underset{\gamma_{th}\rightarrow 0} \sim  P_f \left (X_1 \leq \gamma_{th} \right )\\
&  \underset{\gamma_{th}\rightarrow 0}\sim \frac{\sqrt{\Sigma_{11}}}{\sqrt{2 \pi} \log \left (\frac{1}{\gamma_{th}} \right )}\exp \left (-\frac{\left (\log \left ( \gamma_{th}\right )-\mu_1 \right )^2}{2\Sigma_{11}} \right ).
\end{align}
Moreover, under assumption A we have that $\bar{n}=1$ and the solution of the minimization problem (\ref{mini}) is given by 
\begin{align}
\nonumber \bar{w}_1 &=1 ,\\
\bar{w}_j& =0 \text{  for all  }j \neq 1.
\end{align}
Under assumption A, the value of $\bold{\Lambda}^*$ in (\ref{opt_lam}) simplifies to
\begin{align}\label{lam}
\nonumber \Lambda_1^* & =\log \left (\gamma_{th} \right )-\mu_1,\\
\Lambda_k^* &= \frac{\Sigma_{k1}}{\Sigma_{11}} \left (\log \left (\gamma_{th} \right )-\mu_1 \right ) \text{  for all   }k \neq 1.
\end{align}
Now, we present how the control variate technique could be combined with the previous IS scheme in order to achieve a further variance reduction. 
The control variable $Z_{\gamma_{th}} \left (\bold{Y} \right )$ is selected as follows:
\begin{align}\label{control_var}
Z_{\gamma_{th}}\left (\bold{Y} \right )=\bold{1}_{\left (\exp(Y_1) \leq \gamma_{th} \right )} L \left ( Y_1,Y_2,\cdots, Y_N\right ),
\end{align}
where the random vector $\bold{Y}$ is distributed according to $g(\cdot)$. The expected value $P(\gamma_{th})$ of the control variable $Z_{\gamma_{th}} \left ( \bold {Y}\right )$ is given through straightforward computation by
\begin{align}
\nonumber P (\gamma_{th}) &=\mathbb{E}_{g} \left [ Z_{\gamma_{th}} \left ( \bold {Y}\right )\right ]\\
\nonumber &= P_f \left (\exp(Y_1) \leq \gamma_{th} \right )\\
&= \Phi \left (\frac{\log \left (\gamma_{th} \right )-\mu_1}{\sqrt{\Sigma_{11}}} \right ),
\end{align}
where $\Phi (\cdot)$ is the cumulative distribution function of the standard normal distribution. The main idea of the above choice of the control variable is that, with the values of $\bold{\Lambda^*}$ in (\ref{lam}) and under the PDF $g(\cdot)$, $X_1$ still represents the dominant component in the sense that it determines the asymptotic behavior of the left-tail of the sum. Hence, it is likely that each realizations that belongs to the set $\{X_1 \leq \gamma_{th}\}$  will belong to the set $\{\sum_{i=1}^{N}{X_i} \leq \gamma_{th}\}$ for a sufficiently small threshold value $\gamma_{th}$.

The covariance between $T_{\gamma_{th}} \left ( \bold{Y}\right )$ and $Z_{\gamma_{th}} \left ( \bold{Y}\right )$ which is useful in the computation of the optimal value $\beta^*$ in (\ref{opt_beta}) is given by
\begin{align}
\nonumber &\mathrm{cov}_{g} \left [T_{\gamma_{th}} \left ( \bold{Y}\right ),Z_{\gamma_{th}} \left ( \bold{Y}\right ) \right ]\\
\nonumber &= \mathbb{E}_{g} \left [\bold{1}_{\left ( \sum_{i=1}^{N}{\exp (Y_i)} \leq \gamma_{th}\right )} \bold{1}_{\left (\exp (Y_1) \leq \gamma_{th} \right )} L^2 \left (Y_1,Y_2,\cdots,Y_N \right ) \right ] -\alpha(\gamma_{th}) P(\gamma_{th})\\
\nonumber &= \mathbb{E}_{g} \left [\bold{1}_{\left ( \sum_{i=1}^{N}{\exp (Y_i)} \leq \gamma_{th}\right )}L^2 \left (Y_1,Y_2,\cdots,Y_N \right ) \right ]-\alpha(\gamma_{th})P(\gamma_{th})\\
&= \mathrm{var}_{g} \left [T_{\gamma_{th}} \left ( \bold{Y}\right ) \right ]+\alpha^2(\gamma_{th})-\alpha(\gamma_{th})P(\gamma_{th}).
\end{align}
The variance of $Z_{\gamma_{th}} \left ( \bold{Y}\right )$ is given in a closed-form expression. In fact, via a simple computation, we have 
\begin{align}\label{sec_mom_z}
\nonumber &\mathbb{E}_{g} \left [Z_{\gamma_{th}}^2 \left (\bold{Y} \right )\right ]=\mathbb{E}_{f} \left [ \bold{1}_{\left (\exp \left ( Y_1\right ) \leq \gamma_{th} \right )} L \left (Y_1,Y_2,\cdots, Y_N \right ) \right ]\\
\nonumber &= \exp \left (\frac{1}{2} \bold{\Lambda^*}^t \bold{\Sigma}^{-1} \bold{\Lambda^*}+\bold{\Lambda^*}^t \bold{\Sigma}^{-1} \boldsymbol{\mu} \right ) \int_{\{\exp(y_1) \leq \gamma_{th} \}}{ \exp \left ( -\bold{\Lambda^*}^t \bold{\Sigma}^{-1} \boldsymbol{y}\right ) f(\boldsymbol{y})}dy1 \cdots dy_N\\
&= \exp \left ( \bold{\Lambda^*}^t \bold{\Sigma}^{-1} \bold{\Lambda^*}\right ) P_f \left ( \exp \left (Y_1-\Lambda_{1}^* \right ) \leq \gamma_{th}\right ).
\end{align}
Now, using the value of $\Lambda_1^*$ in (\ref{lam}), we get
\begin{align}\label{egz2}
\mathbb{E}_{g} \left [Z_{\gamma_{th}}^2 \left (\bold{Y} \right )\right ]= \exp \left ( \bold{\Lambda^*}^t \bold{\Sigma}^{-1} \bold{\Lambda^*}\right ) \Phi \left (\frac{2 \left (\log \left ( \gamma_{th}\right )-\mu_1 \right )}{\sqrt{\Sigma_{11}}} \right ).
\end{align}
The following lemma is very useful to study the performance of the estimator $T^{'}_{\gamma_{th}} \left (\bold{Y} \right )$
\begin{lem}
\hspace{2mm}There exists a constant $C_1$ such that
\begin{align}
\frac{\mathbb{E}_{g} \left [Z_{\gamma_{th}}^2 \left (\bold{Y} \right )\right ]-\mathbb{E}_{g} \left [T_{\gamma_{th}}^2 \left (\bold{Y} \right )\right ]}{\mathbb{E}_{g} \left [Z_{\gamma_{th}}^2 \left (\bold{Y} \right )\right ]} \leq C_1\sqrt{\log \left (\frac{1}{\gamma_{th}} \right )} \sqrt {P_1-P_2}.
\end{align}
with $P_1=\mathbb{E}_{g} \left [\bold{1}_{\left (\exp (Y_1) \gamma_{th} \right )} \right ]=\frac{1}{2}$ and $P_2=\mathbb{E}_{g} \left [\bold{1}_{\left ( \sum_{i=1}^{N}{\exp \left (Y_i \right )} \leq \gamma_{th} \right )} \right ]$.
\end{lem}

\begin{proof}
Via Cauchy Schwartz inequality, it follows that
\begin{align}\label{cauchy}
\nonumber &\mathbb{E}_{g} \left [Z_{\gamma_{th}}^2 \left (\bold{Y} \right )\right ]-\mathbb{E}_{g} \left [T_{\gamma_{th}}^2 \left (\bold{Y} \right )\right ]= \mathbb{E}_{g} \left [ L^2 \left ( \bold{1}_{\left ( \exp \left ( Y_1 \right ) \leq \gamma_{th}\right )} -\bold{1}_{\left (\sum_{i=1}^{N}{\exp \left (Y_i \right )} \leq \gamma_{th}\right )}\right )\right ]\\
& \leq \sqrt{\mathbb{E} \left [L^4 \bold{1}_{ \left (\exp \left ( Y_1\right ) \leq \gamma_{th} \right )} \right ]} \sqrt{P_1-P_2} .
\end{align}
Now, let us compute $\mathbb{E} \left [L^4 \bold{1}_{ \left (\exp \left ( Y_1\right ) \leq \gamma_{th} \right )} \right ]$. Using a similar computation as (\ref{sec_mom_z}), we get
\begin{align}\label{egL4}
\nonumber & \mathbb{E}_{g} \left [L^4 \bold{1}_{ \left (\exp \left ( Y_1\right ) \leq \gamma_{th} \right )} \right ]=\mathbb{E}_{f} \left [L^3 \bold{1}_{ \left (\exp \left ( Y_1\right ) \leq \gamma_{th} \right )} \right ]\\
\nonumber &=\exp \left ( \frac{3}{2} \bold{\Lambda^*}^{t} \bold{\Sigma}^{-1}\bold{\Lambda^*}+3 \bold{\Lambda^*}^{t}\bold{\Sigma}^{-1}\boldsymbol{\mu}\right ) \int_{\{ \exp \left ( y_1\right )\leq \gamma_{th}\}}{\exp \left (-3 \bold{\Lambda^*}^{t}\bold{\Sigma}^{-1}\boldsymbol{y} \right ) f(\boldsymbol{y})dy_1 \cdots dy_N}\\
\nonumber &= \exp \left ( \frac{3}{2} \bold{\Lambda^*}^{t} \bold{\Sigma}^{-1}\bold{\Lambda^*}+3 \bold{\Lambda^*}^{t}\bold{\Sigma}^{-1}\boldsymbol{\mu}-\frac{1}{2} \boldsymbol{\mu}^t\bold{\Sigma}^{-1}\boldsymbol{\mu}+\frac{1}{2}(\boldsymbol{\mu}-3\bold{\Lambda^*})^t\bold{\Sigma}^{-1} (\boldsymbol{\mu}-3 \bold{\Lambda^*})\right )\\
\nonumber & \times \int_{\{\exp \left (y_1 \right )\leq \gamma_{th}\}}{f(\boldsymbol{y}+3\bold{\Lambda^*})} dy_1 \cdots dy_N\\
&= \exp \left ( 6 \bold{\Lambda^*}^t \bold{\Sigma}^{-1} \bold{\Lambda^*}\right ) \Phi \left (\frac{4 \left [\log \left (\gamma_{th} \right )-\mu_1 \right ]}{\sqrt{\Sigma_{11}}} \right ).
\end{align}
Now, we use the asymptotic behavior of $\Phi (\cdot)$ \cite{DBLP:journals/anor/AsmussenBJR11}
\begin{align}
\Phi(x) \sim \frac{1}{\sqrt{2\pi}(-x)}\exp \left ( -\frac{1}{2}  x^2\right )  \text{   as   } x \rightarrow -\infty.
\end{align}
By combining (\ref{egz2}), (\ref{cauchy}), and (\ref{egL4}), we get
\begin{align}
\frac{\mathbb{E}_{g} \left [Z_{\gamma_{th}}^2 \left (\bold{Y} \right )\right ]-\mathbb{E}_{g} \left [T_{\gamma_{th}}^2 \left (\bold{Y} \right )\right ]}{\mathbb{E}_{g} \left [Z_{\gamma_{th}}^2 \left (\bold{Y} \right )\right ]} \leq \frac{\exp \left (3 \bold{\Lambda^*}^{t} \bold{\Sigma}^{-1} \bold{\Lambda^*} \right ) \sqrt{\Phi \left (\frac{4 \left [\log \left (\gamma_{th} \right )-\mu_1 \right ]}{\sqrt{\Sigma_{11}}} \right )}\sqrt{P_1-P_2}}{\exp \left (\bold{\Lambda^*}^{t} \bold{\Sigma}^{-1}\bold{\Lambda^*} \right )\Phi \left ( \frac{2 \left [\log \left (\gamma_{th} \right )-\mu_1 \right ]}{\Sigma_{11}}\right )}.
\end{align}
From the expression of $\bold{\Lambda^*}$ in (\ref{lam}), we observe that $\bold{\Lambda^*}^{t} \bold{\Sigma}^{-1}\bold{\Lambda^*}=\frac{(\log (\gamma_{th})-\mu_1)^2}{\Sigma_{11}}$. Hence, it follows that
\begin{align}
\nonumber &\frac{\mathbb{E}_{g} \left [Z_{\gamma_{th}}^2 \left (\bold{Y} \right )\right ]-\mathbb{E}_{g} \left [T_{\gamma_{th}}^2 \left (\bold{Y} \right )\right ]}{\mathbb{E}_{g} \left [Z_{\gamma_{th}}^2 \left (\bold{Y} \right )\right ]} \leq C_1 \sqrt{P_1-P_2}\sqrt{\log \left ( \frac{1}{\gamma_{th}}\right )}\\
& \times \frac{\exp \left (2\frac{(\log (\gamma_{th})-\mu_1)^2}{\Sigma_{11}} \right ) \exp \left (-\frac{4}{\Sigma_{11}} \left ( \log(\gamma_{th})-\mu_1\right )^2 \right )}{\exp \left ( -\frac{2}{\Sigma_{11}} \left (\log \left ( \gamma_{th}\right ) -\mu_1\right )^2\right )}.
\end{align}
Thus, the proof is concluded.
\end{proof}

In the next lemma, we study the asymptotic behavior of $P_1-P_2$.
\begin{lem}
\hspace{2mm} For $i \in \{2,\cdots,N\}$, let us define $a_i=\frac{\Sigma_{i1}}{\Sigma_{11}}-1$ and $c_i=\exp \left (\mu_i -\frac{\Sigma_{i1}}{\Sigma_{11}}\mu_1\right )$ with $a_i>0$ from Assumption A. Let $c_{i_0}\gamma_{th}^{a_{i_0}}=\max_{i}{c_i \gamma_{th}^{a_i}}$, then, for a sufficiently small $\gamma_{th}$, there a constant $C_2$ such that
\begin{align}
P_1-P_2 \leq C_2 \gamma_{th}^{a_{i_0}}.
\end{align}
\end{lem}

\begin{proof}
Let us first re-write $P_1-P_2$ as follows:
\begin{align}
P_1-P_2=P_g\left (\exp(Y_1)\leq \gamma_{th},\sum_{i=1}^{N}{\exp(Y_i)}\geq \gamma_{th} \right ).
\end{align}
Using the value of $\bold{\Lambda^*}$ in (\ref{lam}), the above expression could be expressed as
\begin{align}
P_1-P_2=P \left (\exp(Y_1) \leq 1, \exp (Y_1)+\sum_{i=2}^{N}{c_i \gamma_{th}^{a_i}\exp(Y_i)} \geq 1 \right ),
\end{align}
with $P(\cdot)$ is the probability measure under which $\bold{Y}$ in this case is a multivariate Gaussian vector with zero mean and covariance matrix $\bold{\Sigma}$. Let us denote by $\lambda$ the minimum eigenvalue of $\bold{\Sigma}^{-1}$, then the above expression could be upper-bounded by
\begin{align}
P_1-P_2 \leq d_1 \tilde{P} \left ( \exp(Y_1) \leq 1, \exp(Y_1)+c_{i_0}\gamma_{th}^{a_{i_0}} \sum_{i=2}^{N}{\exp(Y_i)} \geq 1\right ),
\end{align} 
with $d_1=\frac{1}{\sqrt{|\bold{\Sigma}|\lambda^N}}$ and $\tilde{P}(\cdot)$ is the probability measure under which $\bold{Y}$ is now an independent Gaussian random vector with zero mean and covariance matrix $I_{N}/\lambda$ where $I_N$ denotes the identity matrix of order $N$. Note that for a sufficiently small threshold, the index $i_0$ is independent of $\gamma_{th}$. In the other hand, the probability in the right-hand side could be written as
\begin{align}
\nonumber & \tilde{P} \left ( \exp(Y_1) \leq 1, \exp(Y_1)+c_{i_0}\gamma_{th}^{a_{i_0}} \sum_{i=2}^{N}{\exp(Y_i)} \geq 1\right )\\
\nonumber &=\tilde{P} \left ( \exp(Y_1) \leq 1, \exp(Y_1)+c_{i_0}\gamma_{th}^{a_{i_0}} \sum_{i=2}^{N}{\exp(Y_i)} \geq 1, c_{i_0}\gamma_{th}^{a_{i_0}} \sum_{i=2}^{N}{\exp(Y_i)} \leq 1 \right )\\
\nonumber & +\tilde{P} \left ( \exp(Y_1) \leq 1, \exp(Y_1)+c_{i_0}\gamma_{th}^{a_{i_0}} \sum_{i=2}^{N}{\exp(Y_i)} \geq 1, c_{i_0}\gamma_{th}^{a_{i_0}} \sum_{i=2}^{N}{\exp(Y_i)} \geq 1\right )\\
&=I_1(\gamma_{th})+I_2(\gamma_{th}).
\end{align}
Let $\tilde {f} (\cdot)$ be the univariate normal PDF with mean zero and variance $1/\lambda$. Then, the quantity $I_1(\gamma_{th})$ is expressed as follows
\begin{align}
\nonumber & I_1(\gamma_{th})\\
\nonumber &=\int_{\{c_{i_0} \gamma_{th}^{a_{i_0}}\sum_{i=2}^{N}{\exp (y_i)} \leq 1\}}{P \left ( \exp(Y_1) \leq 1, \exp(Y_1) \geq 1-c_{i_0}\gamma_{th}^{a_{i_0}} \sum_{i=2}^{N}{\exp(y_i)}\right ) }\\
\nonumber & \times \tilde {f}(y_2)\cdots  \tilde{f}(y_N)\\
\nonumber &=\int_{\{c_{i_0} \gamma_{th}^{a_{i_0}}\sum_{i=2}^{N}{\exp (y_i)} \leq 1\}}{\left [ \frac{1}{2}-\Phi \left ( \sqrt{\lambda} \log \left (1-c_{i_0}\gamma_{th}^{a_{i_0}}\sum_{i=2}^{N}{\exp(y_i)} \right )\right )\right ]} \\
\nonumber & \times \tilde {f}(y_2)\cdots  \tilde{f}(y_N)\\
\nonumber &=\int_{\{c_{i_0} \gamma_{th}^{a_{i_0}}\sum_{i=2}^{N}{\exp (y_i)} \leq \frac{1}{2}\}}{\left [ \frac{1}{2}-\Phi \left ( \sqrt{\lambda} \log \left (1-c_{i_0}\gamma_{th}^{a_{i_0}}\sum_{i=2}^{N}{\exp(y_i)} \right )\right )\right ] }\\
\nonumber & \times \tilde {f}(y_2)\cdots  \tilde{f}(y_N)\\
\nonumber & +\int_{\{ \frac{1}{2} \leq c_{i_0} \gamma_{th}^{a_{i_0}}\sum_{i=2}^{N}{\exp (y_i)} \leq 1\}}{\left [ \frac{1}{2}-\Phi \left ( \sqrt{\lambda} \log \left (1-c_{i_0}\gamma_{th}^{a_{i_0}}\sum_{i=2}^{N}{\exp(y_i)} \right )\right )\right ]}\\
\nonumber & \times  \tilde {f}(y_2)\cdots  \tilde{f}(y_N)\\
\nonumber &\leq \int_{\{c_{i_0} \gamma_{th}^{a_{i_0}}\sum_{i=2}^{N}{\exp (y_i)} \leq \frac{1}{2}\}}{\left [ \frac{1}{2}-\Phi \left ( \sqrt{\lambda} \log \left (1-c_{i_0}\gamma_{th}^{a_{i_0}}\sum_{i=2}^{N}{\exp(y_i)} \right )\right )\right ] }\\
\nonumber & \times \tilde {f}(y_2)\cdots  \tilde{f}(y_N)\\
&+ \tilde{P} \left (\sum_{i=2}^{N}{\exp (Y_i)} \geq \frac{1}{2c_{i_0}\gamma_{th}^{a_{i_0}}}\right ).
\end{align}
Let $Z=\sum_{i=2}^{N}{\exp(Y_i)}$ and $\tilde{f}_{Z}(\cdot)$ its corresponding PDF, then we have
\begin{align}
\nonumber & I_1(\gamma_{th}) \leq \int_{\{c_{i_0} \gamma_{th}^{a_{i_0}}z \leq \frac{1}{2}\}}{\left [ \frac{1}{2}-\Phi \left ( \sqrt{\lambda} \log \left (1-c_{i_0}\gamma_{th}^{a_{i_0}}z \right )\right )\right ] \tilde {f}_Z(z) dz}\\
&+ \tilde{P} \left (Z \geq \frac{1}{2c_{i_0}\gamma_{th}^{a_{i_0}}}\right ).
\end{align}
Let $g_{\gamma_{th}}(z)=\Phi \left ( \sqrt{\lambda} \log \left (1-c_{i_0}\gamma_{th}^{a_{i_0}}z \right )\right )$. Via a simple computation, we prove that for all $z \leq \frac{1}{2 c_{i_0}\gamma_{th}^{a_{i_0}}}$ 
\begin{align}
\left | g_{\gamma_{th}}^{'}(z) \right | \leq \sqrt{\frac{2\lambda}{\pi}} c_{i_0} \gamma_{th}^{a_{i_0}}.
\end{align}
Therefore, it follows that 
\begin{align}\label{i1}
I_1(\gamma_{th}) \leq \sqrt{\frac{2\lambda}{\pi}} c_{i_0} \gamma_{th}^{a_{i_0}} \mathbb{E}_{\tilde {f}_{z}} \left [Z \right ]+\tilde{P} \left (\sum_{i=2}^{N}{\exp (Y_i)} \geq \frac{1}{2c_{i_0} \gamma_{th}^{a_{i_0}}}\right ).
\end{align}
In the other hand, we have
\begin{align}\label{i2}
\nonumber & I_2(\gamma_{th})\\
\nonumber &=\tilde{P} \left ( \exp(Y_1) \leq 1, \exp(Y_1)+c_{i_0}\gamma_{th}^{a_{i_0}} \sum_{i=2}^{N}{\exp(Y_i)} \geq 1, c_{i_0}\gamma_{th}^{a_{i_0}} \sum_{i=2}^{N}{\exp(Y_i)} \geq 1\right )\\
\nonumber &= \tilde{P} \left ( \exp(Y_1) \leq 1, c_{i_0}\gamma_{th}^{a_{i_0}} \sum_{i=2}^{N}{\exp(Y_i)} \geq 1\right )\\
&= \frac{1}{2} \tilde{P} \left ( \sum_{i=2}^{N}{\exp (Y_i)} \geq \frac{1}{c_{i_0}\gamma_{th}^{a_{i_0}}}\right ).
\end{align}
Therefore, by combining (\ref{i1}) and (\ref{i2}), we get
\begin{align}
\nonumber &P_1-P_2 \leq d_1 \Big ( \sqrt{\frac{2\lambda}{\pi}} c_{i_0} \gamma_{th}^{a_{i_0}} \mathbb{E}_{\bar{f}} \left [Z \right ]+ \frac{1}{2} P \left ( \sum_{i=2}^{N}{\exp (Y_i)} \geq \frac{1}{c_{i_0}\gamma_{th}^{a_{i_0}}}\right )\\
& +P \left ( \sum_{i=2}^{N}{\exp (Y_i)} \geq \frac{1}{2c_{i_0}\gamma_{th}^{a_{i_0}}}\right )\Big ).
\end{align}
In the other hand, the asymptotic behavior of the right-tail of the sum of Log-normal variates is given by \cite{RePEc:eee:stapro:v:78:y:2008:i:16:p:2709-2714}
\begin{align}
P \left ( \sum_{i=2}^{N}{\exp (Y_i)} \geq \frac{1}{c_{i_0}\gamma_{th}^{a_{i_0}}}\right ) \sim \frac{c_1}{-\log(\gamma_{th})} \exp \left ( -\frac{\lambda(-\log (c_{i_0}\gamma_{th}^{a_{i_0}}))^2}{2}\right ) \text{  as  } \gamma_{th} \rightarrow 0.
\end{align}
Hence, the proof is concluded.
\end{proof}
The results in Lemma 1 and Lemma 2 serve to study the goodness of the proposed estimator $T_{\gamma_{th}}^{'} \left (\bold{Y} \right )$. The next proposition provides interesting results on the correlation coefficient as well as the optimal value $\beta^*$.

\begin{prop}
\hspace{2mm}The correlation coefficient $\rho_{T_{\gamma_{th}} \left (\bold{Y} \right ),Z_{\gamma_{th}} \left (\bold{Y} \right )}$ between $T_{\gamma_{th}} \left (\bold{Y} \right )$ and $Z_{\gamma_{th}}\left (\bold{Y} \right )$ goes to $1$ whereas the optimal value $\beta^*$ approaches $-1$ as $\gamma_{th} \rightarrow 0$. 
\end{prop}

\begin{proof}
Let us first prove that $\beta^*$ approaches $-1$ as $\gamma_{th}$ goes to zero. In fact, from the expression of $\beta^*$ in (\ref{opt_beta}), we have
\begin{align}
\beta^*=-\frac{\mathbb{E}_{g} \left [T_{\gamma_{th}}^2 \left (\bold{Y} \right ) \right ]-\alpha(\gamma_{th}) P(\gamma_{th})}{\mathrm{var}_{g} \left [Z_{\gamma_{th}} \left ( \bold{Y}\right ) \right ]}.
\end{align}
From the results of Lemma 1 and Lemma 2, we have that $\mathbb{E}_{g} \left [Z_{\gamma_{th}}^2 \left (\bold{Y} \right ) \right ] \sim \mathbb{E}_{g} \left [T_{\gamma_{th}}^2 \left (\bold{Y} \right ) \right ]$ as $\gamma_{th} \rightarrow 0$. Moreover, the results in (\ref{dominant}) and (\ref{egz2}) implies that $\alpha^2/\mathbb{E}_{g} \left [Z_{\gamma_{th}}^2 \left (\bold{Y} \right ) \right ] \rightarrow 0$ as $\gamma_{th}$ goes to zero. In particular, this shows that  $\mathrm{var}_{g} \left [Z_{\gamma_{th}} \left (\bold{Y} \right ) \right ] \sim \mathbb{E}_{g} \left [T_{\gamma_{th}}^2 \left (\bold{Y} \right ) \right ]$ as $\gamma_{th} \rightarrow 0$. By combining these results and the fact that $\alpha(\gamma_{th}) \sim P(\gamma_{th})$, we conclude that $\beta^*$ approaches $-1$ as $\gamma_{th}\rightarrow 0$.

Now, let us prove that $\rho_{T_{\gamma_{th}} \left (\bold{Y} \right ),Z_{\gamma_{th}} \left (\bold{Y} \right )}$ approaches $1$ as $\gamma_{th} \rightarrow 0$. The expression of $\rho_{T_{\gamma_{th}} \left (\bold{Y} \right ),Z_{\gamma_{th}} \left (\bold{Y} \right )}$ is given by
\begin{align}
\rho_{T_{\gamma_{th}} \left (\bold{Y} \right ),Z_{\gamma_{th}} \left (\bold{Y} \right )}=\frac{\mathbb{E}_{g} \left [T_{\gamma_{th}}^2 \left (\bold{Y} \right ) \right ]-\alpha(\gamma_{th}) P(\gamma_{th})}{\sqrt{\mathrm{var}_{g} \left [T_{\gamma_{th}} \left (\bold{Y} \right ) \right ]\mathrm{var}_{g} \left [Z_{\gamma_{th}} \left (\bold{Y} \right ) \right ]}}.
\end{align}
Given that $\mathrm{var}_{g} \left [Z_{\gamma_{th}} \left (\bold{Y} \right ) \right ] \sim \mathrm{var}_{g} \left [T_{\gamma_{th}} \left (\bold{Y} \right ) \right ]$ as $\gamma_{th} \rightarrow 0$ (this follows from Lemma 1 and 2 and the fact that $\alpha^2/\mathbb{E}_{g} \left [Z_{\gamma_{th}}^2 \left (\bold{Y} \right ) \right ] \rightarrow 0$  and $\alpha(\gamma_{th}) \sim P(\gamma_{th})$ as $\gamma_{th}$ goes to $0$), it follows that
\begin{align}
\rho_{T_{\gamma_{th}} \left (\bold{Y} \right ),Z_{\gamma_{th}} \left (\bold{Y} \right )} \sim \frac{\mathbb{E}_{g} \left [T_{\gamma_{th}}^2 \left (\bold{Y} \right ) \right ]-\alpha(\gamma_{th}) P(\gamma_{th})}{\mathrm{var}_{g} \left [Z_{\gamma_{th}} \left ( \bold{Y}\right ) \right ]}.
\end{align}
Hence, following the same arguments as in the first part of the proof, we conclude that $\rho_{T_{\gamma_{th}} \left (\bold{Y} \right ),Z_{\gamma_{th}} \left (\bold{Y} \right )}$ approaches $1$ as $\gamma_{th}$ goes to zero.
\end{proof}
The main conclusions that can be drawn from the above results are the following. Since $\rho$ approaches $1$, the improved estimator is guaranteed to achieve a variance reduction, with respect to the mean shifting IS scheme, that increases as we decrease the threshold value. The second interesting point is that we may select $\beta$ to be equal to $-1$ instead of working with the optimal unknown value $\beta^*$. In fact, in addition to retrieving the same estimator's performances as with $\beta=\beta^*$, we avoid, when working with $\beta=-1$, the approximation of $\beta^*$ using simulated data which might causes some statistical issues especially when the number of samples is not large enough \cite{glasserman2004monte}.  

The next theorem exhibits the main result of the present work. In fact, we prove that the estimator $T_{\gamma_{th}}^{'} \left ( \bold{Y}\right )$  has an asymptotically vanishing relative error. Such a result represents a relevant contribution to the field of left-tail simulation of the sum of Log-normal variates since previous works have just achieved the weaker property; namely the asymptotic optimality criterion.

\begin{specialcasecounter}
The estimator $T_{\gamma_{th}}^{'} \left ( \bold{Y}\right )$ has the asymptotically vanishing relative error property with $\beta=-1$, that is
\begin{align}
\limsup_{\gamma_{th} \rightarrow 0} {\frac{\mathrm{var}_g \left [ T_{\gamma_{th}}^{'} \left ( \bold{Y}\right ) \right ]}{\alpha^2(\gamma_{th})}}=0.
\end{align}
\end{specialcasecounter} 

\begin{proof}
By replacing $\beta=-1$ into (\ref{variance}), it follows that
\begin{align}
\nonumber &\mathrm{var}_{g} \left [T^{'}_{\gamma_{th}} \left ( \bold{Y}\right ) \right ]=\mathrm{var}_{g} \left [T_{\gamma_{th}} \left ( \bold{Y}\right ) \right ]-2 \mathrm{cov}_{g} \left [T_{\gamma_{th}} \left ( \bold{Y}\right ),Z_{\gamma_{th}} \left (\bold{Y} \right ) \right ]+\mathrm{var}_{g} \left [Z_{\gamma_{th}} \left ( \bold{Y}\right ) \right ]\\
&= \mathbb{E}_{g} \left [ Z_{\gamma_{th}}^2 \left ( \bold{Y}\right )\right ]-\mathbb{E}_{g} \left [ T_{\gamma_{th}}^2 \left ( \bold{Y}\right )\right ]+2 \alpha(\gamma_{th})P(\gamma_{th})-\alpha^2(\gamma_{th})-P^2(\gamma_{th}).
\end{align}
Hence, by dividing the above expression by $\alpha^2(\gamma_{th})$ and via the use of Lemma 1 and Lemma 2, we get
\begin{align}
\nonumber & \frac{\mathrm{var}_{g} \left [T^{'}_{\gamma_{th}} \left ( \bold{Y}\right ) \right ]}{\alpha^2(\gamma_{th})}=\frac{\mathbb{E}_{g} \left [ Z_{\gamma_{th}}^2 \left ( \bold{Y}\right )\right ]-\mathbb{E}_{g} \left [ T_{\gamma_{th}}^2 \left ( \bold{Y}\right )\right ]}{\alpha^2(\gamma_{th})}\\
\nonumber &+\frac{2 \alpha(\gamma_{th})P(\gamma_{th})-\alpha^2(\gamma_{th})-P^2(\gamma_{th})}{\alpha^2(\gamma_{th})}\\
\nonumber & \leq \frac{C_3 \exp \left (\frac{3}{\Sigma_{11}} (\log(\gamma_{th})-\mu_1)^2 \right ) \sqrt{\Phi \left ( \frac{4(\log(\gamma_{th})-\mu_1)}{\sqrt{\Sigma_{11}}}\right )}\gamma_{th}^{a_{i_0}/2}}{\alpha^2(\gamma_{th})}+o(1)\\
& \leq C_4 \gamma_{th}^{a_{i_0}/2} \left (\log(\frac{1}{\gamma_{th}}) \right )^{3/2} +o(1).
\end{align}
Hence, the proof is concluded.
\end{proof}

\section{Simulation Results}
In this section, we provide some selected simulation results to validate some of the results proven in our work and also to investigate the amount of variance reduction achieved by employing the control variate technique compared to the IS scheme of \cite{gulisashvili2016}. The problem parameters are given as follows: the covariance matrix of the $4$-dimensional Gaussian vector is given by
\[\bold{\Sigma}=
\begin{bmatrix}
    1 & 2 & 2 & 2 \\
    2 & 5 & 4 & 4 \\
    2 & 4 & 4.5 & 4 \\
    2 & 4 & 4 & 4.5 \\
\end{bmatrix}
\]
whereas the mean vector is $\mu=(4, 4,  4 ,4)^t$. Note that with the above choice of the covariance matrix, assumption A is satisfied and hence our intuition regarding the choice of the control variable in (\ref{control_var}) is applicable. We define now some performance metrics that serve to compare the proposed estimator with that of \cite{gulisashvili2016}. We define the squared coefficient of variation of the estimator $T^{'}_{\gamma_{th}} \left ( \bold{Y}\right )$ as the ratio of its variance to its squared mean
\begin{align}
CV(T^{'}_{\gamma_{th}} \left ( \bold{Y}\right ))=\frac{\mathrm{var}_g \left [T^{'}_{\gamma_{th}} \left ( \bold{Y}\right ) \right ]}{\alpha^2(\gamma_{th})}.
\end{align} 
Similarly, we define the squared coefficient of variation of the estimator $T_{\gamma_{th}} \left ( \bold{Y}\right )$ as 
\begin{align}
CV(T_{\gamma_{th}} \left ( \bold{Y}\right ))=\frac{\mathrm{var}_g \left [T_{\gamma_{th}} \left ( \bold{Y}\right ) \right ]}{\alpha^2(\gamma_{th})}.
\end{align}
The squared coefficient of variation, also referred as squared relative error in some references, is an interesting performance metric that serves to indicate the number of samples needed to achieve a fixed accuracy requirement. In fact, from the central limit theorem, the number of samples should be  proportional to the squared coefficient of variation in order to maintain the width of the confidence interval constant.

We define also the amount of variance reduction between the mean-shifting IS approach of \cite{gulisashvili2016} and our proposed approach as
\begin{align}
\xi= \frac{\mathrm{var}_g \left [ T_{\gamma_{th}} \left (\bold{Y} \right )\right ]}{\mathrm{var}_g \left [T^{'}_{\gamma_{th}} \left ( \bold{Y}\right ) \right ]}.
\end{align}
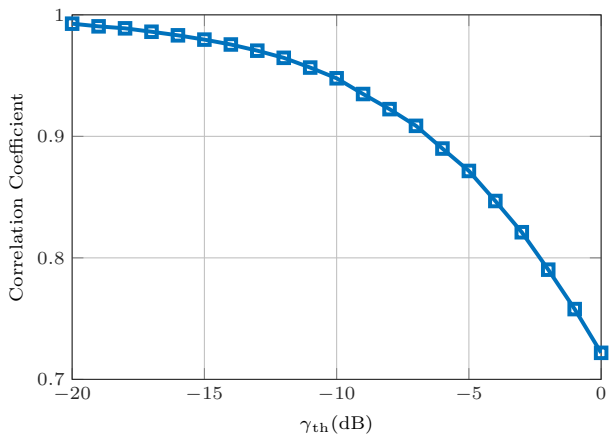
\begin{figure}[h]
\centering
\setlength\figureheight{0.40\textwidth}
\setlength\figurewidth{0.58\textwidth}
%
%
%
%

\definecolor{mycolor1}{rgb}{0,0.447,0.741}
\scalefont{0.7}
\begin{tikzpicture}

\begin{axis}[%
width=\figurewidth,
height=\figureheight,
scale only axis,
every outer x axis line/.append style={darkgray!60!black},
every x tick label/.append style={font=\color{darkgray!60!black}},
xmin=-20, xmax=0,
xlabel={$\gamma{}_{\text{th}}\text{(dB)}$},
xmajorgrids,
every outer y axis line/.append style={darkgray!60!black},
every y tick label/.append style={font=\color{darkgray!60!black}},
ymin=0.7, ymax=1,
ylabel={Correlation Coefficient},
ymajorgrids,
grid style={solid}]
\addplot [
color=mycolor1,
solid,
line width=1.5pt,
mark size=2.0pt,
mark=square,
mark options={solid},
forget plot
]
coordinates{
 (-20,0.992726387761912)(-19,0.99053085733034)(-18,0.988887939904133)(-17,0.986087390790376)(-16,0.983139027199631)(-15,0.979614782142526)(-14,0.975593512376742)(-13,0.970503016813797)(-12,0.964628229500836)(-11,0.95660918823804)(-10,0.947779449261021)(-9,0.934838170665895)(-8,0.922392068744623)(-7,0.908644010475321)(-6,0.889959143580581)(-5,0.871525502275139)(-4,0.846793278151701)(-3,0.821093041640015)(-2,0.79028439751072)(-1,0.75782950722228)(0,0.721854929290764) 
};
\end{axis}
\end{tikzpicture}%
\caption{Correlation Coefficient $\rho_{T_{\gamma_{th}} \left ( \bold{Y}\right ),Z_{\gamma_{th}} \left ( \bold{Y}\right )}$ as a function of $\gamma_{th}$.}
\label{fig1}
\end{figure}
In a first experiment, we aim to validate that the correlation between $T_{\gamma_{th}} \left ( \bold{Y}\right )$ and $Z_{\gamma_{th}} \left ( \bold{Y}\right )$ approaches $1$ as $\gamma_{th}$ goes to $0$. To this end, we plot in Fig. \ref{fig1} the correlation coefficient $\rho_{T_{\gamma_{th}} \left ( \bold{Y}\right ),Z_{\gamma_{th}} \left ( \bold{Y}\right )}$ as a function of $\gamma_{th}$. 
This figure clearly shows that the correlation coefficient $\rho_{T_{\gamma_{th}} \left ( \bold{Y}\right ),Z_{\gamma_{th}} \left ( \bold{Y}\right )}$ approaches $1$ as the event of interest becomes rare and rare, i.e. as we decrease the threshold value $\gamma_{th}$. Such a result supports the performance of the proposed control variate technique in ensuring a variance reduction with respect to the mean shifting approach of \cite{gulisashvili2016} and interestingly this reduction is increasing as we decrease the value of $\gamma_{th}$. 

In a second experiment, we want to investigate the value of $\beta^*$ as a function of the $\gamma_{th}$. We plot in Fig. \ref{fig2} the approximate value of $\beta^*$ as a function of $\gamma_{th}$. From this figure, we easily observe that $\beta^*$ approaches $-1$ as $\gamma_{th}$ decreases which validate the result in Proposition 1. Thus, instead of working with the optimal unknown value $\beta^*$, this result suggests to work with a fixed value of $\beta$ equal to $-1$. 
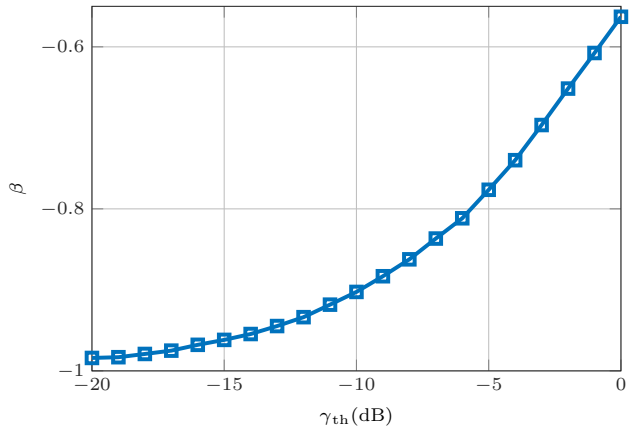
\begin{figure}[h]
\centering
\setlength\figureheight{0.40\textwidth}
\setlength\figurewidth{0.58\textwidth}
%
%
%
%

\definecolor{mycolor1}{rgb}{0,0.447,0.741}
\scalefont{0.7}
\begin{tikzpicture}

\begin{axis}[%
width=\figurewidth,
height=\figureheight,
scale only axis,
every outer x axis line/.append style={darkgray!60!black},
every x tick label/.append style={font=\color{darkgray!60!black}},
xmin=-20, xmax=0,
xlabel={$\gamma{}_{\text{th}}\text{(dB)}$},
xmajorgrids,
every outer y axis line/.append style={darkgray!60!black},
every y tick label/.append style={font=\color{darkgray!60!black}},
ymin=-1, ymax=-0.55,
ylabel={$\beta$},
ymajorgrids,
grid style={solid}]
\addplot [
color=mycolor1,
solid,
line width=1.5pt,
mark size=2.0pt,
mark=square,
mark options={solid},
forget plot
]
coordinates{
 (-20,-0.984040006611881)(-19,-0.983119270969396)(-18,-0.979116486170582)(-17,-0.975007011989505)(-16,-0.967831290800369)(-15,-0.961736996940846)(-14,-0.954503856524289)(-13,-0.944661825161449)(-12,-0.933698937095877)(-11,-0.918255758738925)(-10,-0.902533886498487)(-9,-0.883177397036299)(-8,-0.862305326254788)(-7,-0.83664775505529)(-6,-0.811629747129908)(-5,-0.776437165165523)(-4,-0.739917287766101)(-3,-0.696522796641628)(-2,-0.65152189345814)(-1,-0.607587935913239)(0,-0.562878020225772) 
};
\end{axis}
\end{tikzpicture}%
\caption{Approximate optimal value $\beta^*$ as a function of $\gamma_{th}$.}
\label{fig2}
\end{figure}

In Fig. \ref{fig3}, we plot the estimated value of $\alpha(\gamma_{th})$ as a function of $\gamma_{th}$ using the mean-shifting IS estimator of \cite{gulisashvili2016} as well as our proposed estimator with $\beta=\beta^*$ and $\beta=-1$.
\begin{figure}[h]
\centering
\setlength\figureheight{0.40\textwidth}
\setlength\figurewidth{0.58\textwidth}
%
%
%
%
\begin{tikzpicture}
\scalefont{0.7}
\begin{semilogyaxis}[%
width=\figurewidth,
height=\figureheight,
scale only axis,
every outer x axis line/.append style={darkgray!60!black},
every x tick label/.append style={font=\color{darkgray!60!black}},
xmin=-10, xmax=5,
xlabel={$\gamma{}_{\text{th}}\text{(dB)}$},
xmajorgrids,
every outer y axis line/.append style={darkgray!60!black},
every y tick label/.append style={font=\color{darkgray!60!black}},
ymin=1e-10, ymax=0.01,
yminorticks=true,
ylabel={$\alpha\text{(}\gamma{}_{\text{th}}\text{)}$},
ymajorgrids,
yminorgrids,
grid style={solid},
legend style={at={(0.012177658917024,0.755091779025583)},anchor=south west,draw=darkgray!60!black,fill=white,align=left}]
\addplot [
color=blue,
solid,
line width=1.5pt,
mark size=2.0pt,
mark=o,
mark options={solid},
]
coordinates{
 (-10,1.39078438214206e-10)(-9,5.92831717898207e-10)(-8,2.39891311234351e-09)(-7,9.19339754379149e-09)(-6,3.3178416833905e-08)(-5,1.14246124767071e-07)(-4,3.72563797672772e-07)(-3,1.14711639583681e-06)(-2,3.33728503335302e-06)(-1,9.20105925303787e-06)(0,2.39206315755167e-05)(1,5.92854253468175e-05)(2,0.000138352368010141)(3,0.000305968287807699)(4,0.000640439537511128)(5,0.00127384016515106) 
};
\addlegendentry{$\hat \alpha_{IS}(\gamma_{th})$}
\addplot [
color=red,
solid,
line width=1.5pt,
mark size=2.0pt,
mark=square,
mark options={solid},
]
coordinates{
 (-10,1.3943763615063e-10)(-9,5.9397935938202e-10)(-8,2.40157246500639e-09)(-7,9.19654732849352e-09)(-6,3.32343639862012e-08)(-5,1.14466097974979e-07)(-4,3.72587590610236e-07)(-3,1.14320630342645e-06)(-2,3.33170712785236e-06)(-1,9.19293921944186e-06)(0,2.39404460647169e-05)(1,5.90439396419166e-05)(2,0.000138006465903744)(3,0.000305281856040731)(4,0.000642231060106493)(5,0.00127632862727998) 
};
\addlegendentry{$\hat \alpha_{IS-CV}(\gamma_{th})$ with $\beta=\beta^*$}
\addplot [
color=black,
solid,
line width=1.5pt,
mark size=2.0pt,
mark=diamond,
mark options={solid},
]
coordinates{
 (-10,1.39349423975925e-10)(-9,5.945482402177e-10)(-8,2.40344092357806e-09)(-7,9.18067456943907e-09)(-6,3.33069486712903e-08)(-5,1.14605435987255e-07)(-4,3.72522430077446e-07)(-3,1.1462761488257e-06)(-2,3.33532952539136e-06)(-1,9.18890140224322e-06)(0,2.39853017319104e-05)(1,5.91346461279796e-05)(2,0.000138207174788755)(3,0.000305388535045105)(4,0.000641376634184789)(5,0.00127249825726123) 
};
\addlegendentry{$\hat \alpha_{IS-CV}(\gamma_{th})$ with $\beta=-1$}
\end{semilogyaxis}
\end{tikzpicture}%
\caption{$\alpha(\gamma_{th})$ as a function of $\gamma_{th}$ with $M=10^6$.}
\label{fig3}
\end{figure}
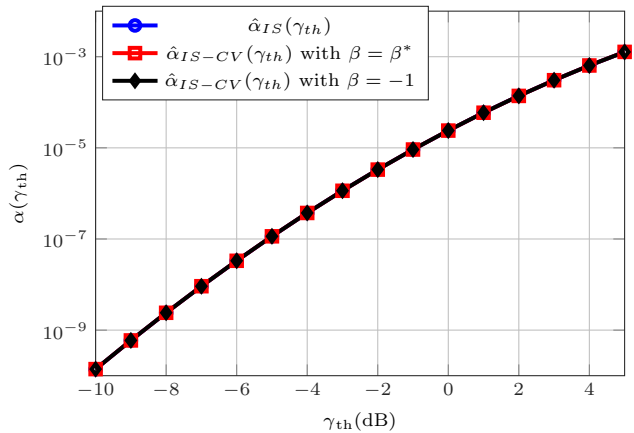
Note that we employ the version of \cite{glasserman2004monte} where $\beta^*$ is estimated using the same simulated data as for the estimation. While this could result in a dependence across the replicants, it was noted in \cite{glasserman2004monte} that when $M$ is sufficiently large, the dependence could be ignored and we retrieve the same performance when using the approximate value of $\beta^*$ instead of the optimal unknown value. From Fig. \ref{fig3}, we observe that the three curves coincide perfectly and thus the three approaches yield accurate estimates of $\alpha(\gamma_{th})$. 
\begin{figure}[h]
\centering
\setlength\figureheight{0.40\textwidth}
\setlength\figurewidth{0.58\textwidth}
%
%
%
%
\begin{tikzpicture}
\scalefont{0.7}
\begin{axis}[%
width=\figurewidth,
height=\figureheight,
scale only axis,
every outer x axis line/.append style={darkgray!60!black},
every x tick label/.append style={font=\color{darkgray!60!black}},
xmin=-10, xmax=5,
xlabel={$\gamma{}_{\text{th}}\text{(dB)}$},
xmajorgrids,
every outer y axis line/.append style={darkgray!60!black},
every y tick label/.append style={font=\color{darkgray!60!black}},
ymin=0, ymax=8,
ylabel={Squared Coefficient of Variation},
ymajorgrids,
grid style={solid},
legend style={at={(0.322176710154079,0.737699596786146)},anchor=south west,draw=darkgray!60!black,fill=white,align=left}]
\addplot [
color=blue,
solid,
line width=1.5pt,
mark size=2.0pt,
mark=o,
mark options={solid},
]
coordinates{
 (-10,7.22680000251191)(-9,6.95850425697925)(-8,6.7077668812122)(-7,6.43624393716178)(-6,6.19114274535568)(-5,5.92701529809103)(-4,5.65792949058653)(-3,5.46100969227461)(-2,5.19616604667289)(-1,4.95391142526849)(0,4.736719828414)(1,4.54720839905784)(2,4.32826302657732)(3,4.12906891501519)(4,3.90861065234556)(5,3.75109386571306) 
};
\addlegendentry{$CV(T_{\gamma_{th}} \left ( \bold{Y}\right ))$}
\addplot [
color=red,
solid,
line width=1.5pt,
mark size=2.0pt,
mark=square,
mark options={solid},
]
coordinates{
 (-10,0.748421189613002)(-9,0.874307846788636)(-8,0.978858321833015)(-7,1.1210900547933)(-6,1.30783469944288)(-5,1.43025776002393)(-4,1.58754242024806)(-3,1.77850756154308)(-2,1.96136654124159)(-1,2.10639336645496)(0,2.28257812994477)(1,2.42582390495309)(2,2.54768711003214)(3,2.64073016984112)(4,2.71019954246967)(5,2.78222256881984) 
};
\addlegendentry{$CV(T^{'}_{\gamma_{th}} \left ( \bold{Y}\right ))$ with $\beta=\beta^*$}
\addplot [
color=black,
solid,
line width=1.5pt,
mark size=2.0pt,
mark=diamond,
mark options={solid},
]
coordinates{
 (-10,0.8251392609235)(-9,0.958778477284653)(-8,1.1164455440655)(-7,1.34316989592952)(-6,1.571975608646)(-5,1.7922920683582)(-4,2.09251381696453)(-3,2.44756884949177)(-2,2.84774832332508)(-1,3.28562709101247)(0,3.78674673705694)(1,4.36935461489926)(2,4.99120979786105)(3,5.70091618484856)(4,6.3868362629318)(5,7.23431731295649) 
};
\addlegendentry{$CV(T^{'}_{\gamma_{th}} \left ( \bold{Y}\right ))$ with $\beta=-1$}
\end{axis}
\end{tikzpicture}%
\caption{Squared coefficient of variation  as a function of $\gamma_{th}(dB)$.}
\label{fig4}
\end{figure}
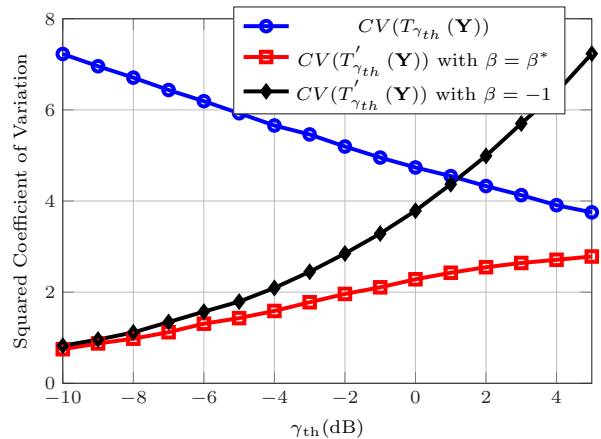

Now, we focus on studying the efficiency of the three approaches. To this end, we plot in Fig. \ref{fig4} the squared coefficient of variation given by the mean-shifting approach and the two versions of the control variate techniques; namely when $\beta$ is equal to the approximate value of $\beta^*$ and when $\beta$ is equal to $-1$. We point out from this figure that the squared coefficient of variation corresponding to the mean-shifting IS scheme is increasing as we decrease the threshold values. This is expected from our analysis since the mean-shifting approach is only asymptotically optimal. More precisely, from (\ref{dominant}) and (\ref{sec_mom_z}), the quantity $CV \left ( T_{\gamma_{th}} \left ( \bold{Y}\right )\right )$ is equivalent to a constant times $\log \left (\frac{1}{\gamma_{th}} \right )$. 
Hence, in order to meet a fixed accuracy requirement, i.e. maintain the width of the confidence interval fixed, the number of simulation runs required by the mean-shifting approach should be of the order of $\log \left (\frac{1}{\gamma_{th}} \right )$.

On the other hand, the squared coefficient of variation of the control variate estimator, using the two values of $\beta$, is decreasing as we decrease the threshold values. Such an observation is expected from the fact that our proposed estimator $T^{'}_{\gamma_{th}} \left (\bold{Y} \right )$ with $\beta=-1$ has the asymptotically vanishing relative error property as it was proven in Theorem 1. 
\begin{figure}[h]
\centering
\setlength\figureheight{0.40\textwidth}
\setlength\figurewidth{0.58\textwidth}
%
%
%
%
\begin{tikzpicture}
\scalefont{0.7}
\begin{axis}[%
width=\figurewidth,
height=\figureheight,
scale only axis,
every outer x axis line/.append style={darkgray!60!black},
every x tick label/.append style={font=\color{darkgray!60!black}},
xmin=-10, xmax=5,
xlabel={$\gamma{}_{\text{th}}\text{(dB)}$},
xmajorgrids,
every outer y axis line/.append style={darkgray!60!black},
every y tick label/.append style={font=\color{darkgray!60!black}},
ymin=0, ymax=10,
ylabel={Variance Reduction},
ymajorgrids,
grid style={solid},
legend style={draw=darkgray!60!black,fill=white,align=left}]
\addplot [
color=red,
solid,
line width=1.5pt,
mark size=2.0pt,
mark=square,
mark options={solid},
]
coordinates{
 (-10,9.65606012070393)(-9,7.9588720180633)(-8,6.85264325959983)(-7,5.74105881114828)(-6,4.73388781318695)(-5,4.14401897598645)(-4,3.5639548388902)(-3,3.07055747771829)(-2,2.64925802363469)(-1,2.35184534102757)(0,2.07516218887483)(1,1.87450061390411)(2,1.69889897763887)(3,1.56360879357228)(4,1.44218556275891)(5,1.34823644511812) 
};
\addlegendentry{$\xi$ with $\beta=\beta^*$}
\addplot [
color=black,
solid,
line width=1.5pt,
mark size=2.0pt,
mark=diamond,
mark options={solid},
]
coordinates{
 (-10,8.7582791714742)(-9,7.25767674373163)(-8,6.00814515035466)(-7,4.79183158933717)(-6,3.93844707977902)(-5,3.30694723406346)(-4,2.70389110204019)(-3,2.23119757934841)(-2,1.82465774946213)(-1,1.50775218490846)(0,1.25086787084562)(1,1.04070481795003)(2,0.86717713778174)(3,0.724281638447712)(4,0.611979153909193)(5,0.518513869856792) 
};
\addlegendentry{$\xi$ with $\beta=-1$}
\end{axis}
\end{tikzpicture}%
\caption{Amount of variance reduction as a function of $\gamma_{th}(dB)$.}
\label{fig5}
\end{figure}
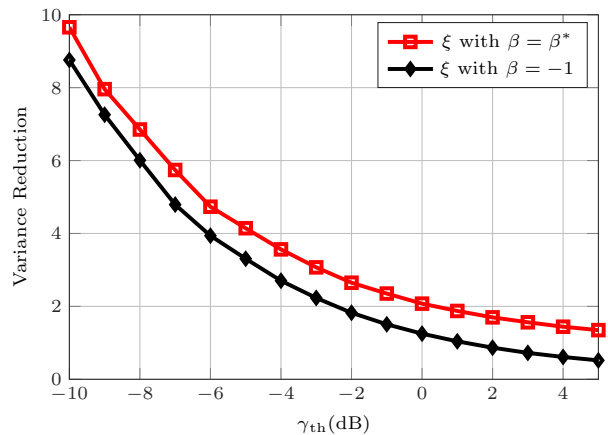
Hence, the number of samples needed by the proposed control variate technique with both values of $\beta$ is getting smaller in order to meet a certain accuracy measured by the width of the confidence interval. Moreover, we point out that asymptotically $CV(T^{'}_{\gamma_{th}} \left ( \bold{Y}\right ))$ with $\beta=-1$ approaches the optimal value of $CV(T^{'}_{\gamma_{th}} \left ( \bold{Y}\right ))$ which in accordance with the result proven in  Proposition 1.

Finally, we aim to quantify the amount of variance reduction $\xi$ achieved by both versions of the proposed control variate estimators compared to the mean-shifting one. We plot then in Fig. \ref{fig5} the value of $\xi$ as a function of $\gamma_{th}$. Interestingly, we observe that, for both versions of the control variate approaches, the amount of variance reduction is increasing as decrease the probability of interest. Moreover, the value of $\xi$ with $\beta=-1$ approaches, as we decrease $\gamma_{th}$, the optimal value corresponding to $\beta=\beta^*$ which is $\frac{1}{1-\rho_{T_{\gamma_{th}} \left ( \bold{Y}\right ),Z_{\gamma_{th}} \left ( \bold{Y}\right )}^2}$.

\section{Conclusion}
In this paper, we considered the problem of estimating the probability that a sum of correlated Log-normal variates with Gaussian copula falls below a certain threshold. We developed a variance reduction technique based on a combination of a control variate approach with a previously developed mean-shifting importance sampling approach. Under a mild assumption on the covariance matrix of the corresponding multivariate Gaussian random vector, we proved that the proposed estimator has the asymptotically vanishing relative error property. This result is a major contribution in the field of left-tail estimation of sum of Log-normal random variables since previous estimators were only  proven to satisfy weaker properties. We showed some selected simulations in order to validate our theoretical results and to quantify the performances of the proposed estimator with and without the employment of the control variate technique. 
\bibliography{References_arxiv}

\begin{thebibliography}{10}

\bibitem{DBLP:journals/anor/AsmussenBJR11}
S.~Asmussen, J.~H. Blanchet, S.~Juneja, and L.~Rojas{-}Nandayapa.
\newblock Efficient simulation of tail probabilities of sums of correlated
  {L}ognormals.
\newblock {\em Annals {OR}}, 189(1):5--23, 2011.

\bibitem{opac-b1123521}
S.~Asmussen and P.~W. Glynn.
\newblock {\em Stochastic simulation : algorithms and analysis}.
\newblock Stochastic modelling and applied probability. Springer, New York,
  2007.

\bibitem{Laplace}
S.~Asmussen, J.~L. Jensen, and L.~Rojas-Nandayapa.
\newblock On the {L}aplace transform of the {L}ognormal distribution.
\newblock {\em Methodology and Computing in Applied Probability}, pages 1--18,
  2014.

\bibitem{asmussen2016exponential}
S.~Asmussen, J.~L. Jensen, and L.~Rojas-Nandayapa.
\newblock Exponential family techniques for the lognormal left tail.
\newblock {\em Scandinavian Journal of Statistics}, 43(3):774--787, Sep. 2016.

\bibitem{reference_1}
S.~Asmussen and D.~P. Kroese.
\newblock Improved algorithms for rare event simulation with heavy tails.
\newblock {\em Advances in Applied Probability}, pages 545--558, 2006.

\bibitem{RePEc:eee:stapro:v:78:y:2008:i:16:p:2709-2714}
S.~Asmussen and L.~Rojas-Nandayapa.
\newblock Asymptotics of sums of lognormal random variables with gaussian
  copula.
\newblock {\em Statistics \& Probability Letters}, 78(16):2709--2714, 2008.

\bibitem{1275712}
N.~C. Beaulieu and Q.~Xie.
\newblock An optimal {L}ognormal approximation to {L}ognormal sum
  distributions.
\newblock {\em IEEE Transactions on Vehicular Technology}, 53(2):479--489,
  {M}ar. 2004.

\bibitem{BenRached2016}
N.~Ben~Rached, F.~Benkhelifa, A.~Kammoun, M.-S. Alouini, and R.~Tempone.
\newblock On the generalization of the hazard rate twisting-based simulation
  approach.
\newblock {\em Statistics and Computing}, pages 1--15, 2016.

\bibitem{7328688}
N.~Ben~Rached, A.~Kammoun, M.-S. Alouini, and R.~Tempone.
\newblock Unified importance sampling schemes for efficient simulation of
  outage capacity over generalized fading channels.
\newblock {\em IEEE Journal of Selected Topics in Signal Processing},
  10(2):376--388, Mar. 2016.

\bibitem{DBLP:conf/wsc/BlanchetJR08}
J.~H. Blanchet, S.~Juneja, and L.~Rojas{-}Nandayapa.
\newblock Efficient tail estimation for sums of correlated lognormals.
\newblock In {\em Proceedings of the 2008 Winter Simulation Conference, Miami,
  Florida, USA, December 7-10, 2008}, pages 607--614, 2008.

\bibitem{4784348}
M.~Di~Renzo, F.~Graziosi, and F.~Santucci.
\newblock Approximating the linear combination of {L}og-normal {RV}s via
  {P}earson type {IV} distribution for {UWB} performance analysis.
\newblock {\em IEEE Transactions on Communications}, 57(2):388--403, {F}eb.
  2009.

\bibitem{4814351}
M.~Di~Renzo, F.~Graziosi, and F.~Santucci.
\newblock Further results on the approximation of {L}og-normal power sum via
  {P}earson type {IV} distribution: {A} general formula for log-moments
  computation.
\newblock {\em IEEE Transactions on Communications}, 57(4):893--898, {A}pr.
  2009.

\bibitem{CIS-230627}
D.~Dufresne.
\newblock The log-normal approximation in financial and other computations.
\newblock {\em Advances in Applied Probability}, 36(3):747--773, 2004.

\bibitem{citeulike:6297231}
L.~Fenton.
\newblock The sum of {L}og-normal probability distributions in scatter
  transmission systems.
\newblock {\em IRE Transactions on Communications Systems}, 8(1):57--67, 1960.

\bibitem{Ghavami}
M.~Ghavami, R.~Kohno, and L.~Michael.
\newblock {\em Ultra Wideband Signals and Systems in Communication
  Engineering}.
\newblock Wiley, Chichester, 2004.

\bibitem{glasserman2004monte}
P.~Glasserman.
\newblock {\em Monte Carlo Methods in Financial Engineering}.
\newblock Springer, New York, 2004.

\bibitem{gulisashvili2016}
A.~Gulisashvili and P.~Tankov.
\newblock Tail behavior of sums and differences of log-normal random variables.
\newblock {\em Bernoulli}, 22(1):444--493, Feb. 2016.

\bibitem{Juneja:2002:SHT:566392.566394}
S.~Juneja and P.~Shahabuddin.
\newblock Simulating heavy tailed processes using delayed hazard rate twisting.
\newblock {\em ACM Trans. Model. Comput. Simul.}, 12(2):94--118, {A}pr. 2002.

\bibitem{opac-b1132466}
D.~P. Kroese, T.~Taimre, and Z.~I. Botev.
\newblock {\em Handbook of Monte Carlo methods}.
\newblock Wiley, N.J, 2011.

\bibitem{journals/twc/NavidpourUK07}
S.~M. Navidpour, M.~Uysal, and M.~Kavehrad.
\newblock {BER} performance of free-space optical transmission with spatial
  diversity.
\newblock {\em IEEE Transactions on Wireless Communications}, 6(8):2813--2819,
  {A}ug. 2007.

\bibitem{CIS-358019}
L.~Rojas-Nandayapa and J.~H. Blanchet.
\newblock Efficient simulation of tail probabilities of sums of dependent
  random variables.
\newblock {\em Journal of Applied Probability}, 2011:--, 2011.

\bibitem{179349}
J.S. Sadowsky.
\newblock On the optimality and stability of exponential twisting in {M}onte
  {C}arlo estimation.
\newblock {\em IEEE Transactions on Information Theory}, 39(1):119--128, Jan.
  1993.

\bibitem{citeulike:7151841}
S.~C. Schwartz and Y.~S. Yeh.
\newblock {On the distribution function and moments of power sums with
  {L}ognormal component}.
\newblock {\em The Bell Systems Technical Journal}, 1982.

\bibitem{Stuber:2001:PMC:368633}
G.~L. St\"{u}ber.
\newblock {\em Principles of Mobile Communication, 2nd Edition.}
\newblock Kluwer Academic Publishers, Norwell, MA, USA, 2001.

\end{thebibliography}
\bibliographystyle{plain}
\end{document}